\documentclass[12pt, a4paper]{article}
\usepackage{float}
\usepackage{multirow,diagbox}
\usepackage{amssymb,amsmath, amsfonts, booktabs,bbm}
\usepackage{lscape}
\usepackage{enumerate}
\usepackage{amsthm}
\usepackage[T1]{fontenc}

\usepackage[utf8]{inputenc}
\usepackage{lmodern}
\usepackage[english]{babel}
\usepackage{csquotes}
\usepackage[colorlinks=true, linkcolor=blue]{hyperref}
\usepackage{mathtools}
\usepackage{xcolor}
\numberwithin{equation}{section}
\usepackage{graphicx}
\usepackage{caption}
\usepackage{subcaption}
\usepackage{upgreek}
\newtheorem{proposition}{Proposition}[section]

\newtheorem{definition}{Definition}[section]

\begin{document}
	\newcommand{\bea}{\begin{eqnarray}}
		\newcommand{\eea}{\end{eqnarray}}
	\newcommand{\nn}{\nonumber}
	\newcommand{\bee}{\begin{eqnarray*}}
		\newcommand{\eee}{\end{eqnarray*}}
	\newcommand{\lb}{\label}
	\newcommand{\nii}{\noindent}
	\newcommand{\ii}{\indent}
	\newtheorem{theorem}{Theorem}[section]
	\newtheorem{example}{Example}[section]
	\numberwithin{equation}{section}
	\renewcommand{\qedsymbol}{\rule{0.7em}{0.7em}}
	\renewcommand{\theequation}{\thesection.\arabic{equation}}
	%\bibpunct[, ]{(}{)}{;}{a}{,}{,}
	%\doublespacing	
	\begin{center}
	{\Large\bf Copula-Based Modeling of Fractional Inaccuracy: A Unified Framework} \\
	\vspace{0.3in}
	
	{\bf Aman Pandey~~~~ \bf Chanchal Kundu$^{a}$\footnote{}}
	\\
	Department of Mathematical Sciences\\ Rajiv Gandhi Institute of Technology\\Jais, Raebareli, 229304\\
\end{center}
\date{}

		\begin{center}
\textbf{Abstract}
		\end{center} 
We introduce novel information-theoretic measures termed the multivariate cumulative copula fractional inaccuracy measure and the multivariate survival copula fractional inaccuracy measure, constructed respectively from multivariate copulas and multivariate survival copulas. These measures generalize the concept of fractional inaccuracy to multivariate settings by incorporating dependence structures through copulas. We establish bounds for these measures using the Fréchet–Hoeffding bounds and investigate their behavior under lower and upper orthant stochastic orderings to facilitate comparative analysis. Furthermore, we define the multivariate co-copula fractional inaccuracy measure and the multivariate dual copula fractional inaccuracy measure, derived from the multivariate co-copula and dual copula, respectively, and examine several analogous properties for these extended forms.
		 \\
\textbf{MSCs:} 94A17; 62B10; 60E15.
\section{Introduction}
In information theory, various uncertainty and divergence measures have been extensively studied due to their wide-ranging applications across disciplines such as statistics, reliability theory, survival analysis, and copula modeling. In today’s era—aptly referred to as the “information age”—systems that enable the generation, storage, transmission, and analysis of information are pivotal in virtually every domain. Consequently, quantifying information and understanding the nature of uncertainty have become essential for modeling, inference, and decision-making across scientific and engineering fields.

Information theory, which emerged from the foundational work of Shannon~\textcolor{blue}{(1948)}, provides a rigorous mathematical framework to measure uncertainty and information content. For an absolutely continuous random variable $X$ with density $f(x)$ and support $A_X$, the Shannon entropy is defined as:
\begin{align}
	H(X) = -\int_{A_X} f(x) \ln f(x) \, dx,
\end{align}
which captures the average amount of information (or uncertainty) inherent in the distribution of $X$. While Shannon entropy is foundational, it suffers from several limitations—it is only defined when $X$ admits a density function, and its value may lie on the extended real line, complicating its interpretation in certain settings.

To address these drawbacks, alternative entropy measures have been proposed. Rao et al.~\textcolor{blue}{(2004)} introduced the \textit{cumulative residual entropy} (CRE), defined for nonnegative random variables as:
\begin{align}
	\mathcal{H}(X) = -\int_0^\infty \overline{F}(x) \ln \overline{F}(x) \, dx,
\end{align}
where $\overline{F}(x)$ is the survival function of $X$. This formulation avoids reliance on the density function and instead integrates cumulative information, making it suitable for systems characterized by remaining lifetimes and reliability-based analysis. The CRE has found practical applications in signal processing, machine learning, and reliability engineering(Rao, \textcolor{blue}{2005}).

An analogous measure, the \textit{cumulative past entropy} (CPE), was introduced by Di Crescenzo and Longobardi~\textcolor{blue}{(2009)} as:
\begin{align}
	\mathcal{H}(X) = -\int_0^\infty F(x) \ln F(x) \, dx,
\end{align}
which focuses on the elapsed or past lifetime. Both CRE and CPE provide alternative descriptions of uncertainty based on the cumulative behavior of distributions and are particularly useful when the instantaneous density is inaccessible or ill-defined.

In parallel, the concept of inaccuracy, originally formulated by Kerridge~\textcolor{blue}{(1961)}, aims to quantify the divergence between a true distribution and a proposed model. Given two absolutely continuous nonnegative random variables $X$ and $Y$ with densities $f_X$ and $f_Y$, respectively, the inaccuracy measure is:
\begin{align}
	I(X, Y) = -\int_0^\infty f_X(x) \ln f_Y(x) \, dx = \mathbb{E}_{f_X}[-\ln f_Y(X)],
\end{align}
which is minimized when the assigned model $f_Y$ matches the true distribution $f_X$. This measure underpins the widely-used Kullback–Leibler divergence (Kullback,\textcolor{blue} {1951}) and plays a critical role in hypothesis testing, model selection, and information geometry.

Extending the idea of cumulative entropy to inaccuracy settings, Kundu et al.~\textcolor{blue}{(2017)} introduced the \textit{cumulative residual inaccuracy} (CRI) measure as:
\begin{align}
	\mathcal{K}(X) = -\int_0^\infty \overline{F}(x) \ln \overline{G}(x) \, dx,
\end{align}
where $\overline{F}(x)$ and $\overline{G}(x)$ are the survival functions of $X$ and $Y$, respectively. Similarly, the \textit{cumulative past inaccuracy} (CPI) is defined by:
\begin{align}
	\overline{\mathcal{K}}(X) = -\int_0^\infty F(x) \ln G(x) \, dx. \tag{1.6}
\end{align}
These cumulative inaccuracy measures preserve several important properties, including non-negativity, monotonicity under transformations, and interpretability in reliability and risk settings.

To generalize these measures within the framework of fractional calculus, Kharazmi and Contreras~\textcolor{blue}{(2024)} proposed the \textit{fractional cumulative residual inaccuracy} (FCRI) as:
\begin{align}
	\mathcal{K}(X) = -\int_0^\infty \overline{F}(x) \left(-\log \overline{G}(x)\right)^\delta \, dx,
\end{align}
which introduces a fractional order $\delta$ to modulate the contribution of logarithmic divergence. This formulation enables enhanced sensitivity to tail behavior and long-range dependencies, and has been employed in the analysis of chaotic systems such as Chebyshev and Logistic maps, where small perturbations can result in significant structural changes.

A central analytical tool in fractional calculus is the \textit{Mittag-Leffler function} (MLF), introduced by Mittag-Leffler~\textcolor{blue}{(1903)} and defined as:
\begin{align}
	E_\delta(x) = \sum_{r=0}^\infty \frac{x^r}{\Gamma(\delta r + 1)}, \quad 0 < \delta < 1,
\end{align}
which generalizes the exponential function and arises in solutions to fractional differential and integral equations (Haubold et al., \textcolor{blue}{2011}). The inverse MLF, denoted $\text{Ln}_\delta(x)$, satisfies the functional identity (Jumarie, \textcolor{blue}{2012})
\begin{align}
	g(uv) = g(u) + g(v), \quad u, v > 0,
\end{align}
and its transformed version $\left(\text{Ln}_\delta x\right)^{1/\delta}$ adheres to a generalized additive rule. Although closed-form expressions are not available, these functions admit useful numerical approximations and are applicable in nonlocal modeling.

Rewriting the FCRI in terms of the inverse MLF yields:
\begin{align}
	\mathcal{K}(X) = \int_0^\infty \overline{F}(x) \left(-\text{Ln}_\delta \overline{G}(x)\right)^{1/\delta} \, dx, \quad 0 < \delta < 1.
\end{align}
Saha and Kayal~\textcolor{blue}{(2023)} further developed this idea by introducing the \textit{fractional cumulative past inaccuracy} (FCPI) measure using:
\begin{align}
	\mathcal{K}(X) = \int_0^\infty F(x) \left(-\text{Ln}_\delta G(x)\right)^{1/\delta} \, dx, \quad 0 < \delta < 1,
\end{align}
and studied its properties under affine transformations and dynamic reliability models.

Recent years have witnessed growing interest in embedding information-theoretic measures within the copula framework to model dependence structures independently of marginal distributions. Copulas provide a powerful tool for representing joint distributions via their marginals and have been employed in various applications from finance to hydrology. Early work by Zhao and Lin~\textcolor{blue}{(2011)} proposed a copula-based entropy framework grounded in Jaynes’s maximum entropy principle. Ma and Sun~\textcolor{blue}{(2011)} introduced a copula-based mutual information measure, while Hao and Singh~\textcolor{blue}{(2013)} applied entropy maximization to simulate multisite streamflow with spatiotemporal dependencies.
Grønneberg and Hjort~\textcolor{blue}{(2014)} modified the Akaike information criterion using copula-based Kullback–Leibler divergence, and Singh and Zhang~\textcolor{blue}{(2018)} derived the most entropic copula through entropy maximization, independent of marginal forms. Fallah et al.~\textcolor{blue}{(2019)} extended these ideas using Tsallis entropy to derive constrained copula structures based on Spearman’s rho.
Complementing entropy-based developments, recent research has focused on copula-based inaccuracy measures. Hosseini and Ahmadi~\textcolor{blue}{(2019)} defined inaccuracy measures using both copula densities and cumulative copulas. These were extended to co-copula and dual copula domains by Hosseini and Nooghabi~\textcolor{blue}{(2021)}, who established desirable mathematical properties such as symmetry and consistency under hazard models. Preda et al.~\textcolor{blue}{(2023)} proposed Tsallis-type inaccuracy measures for dual and co-copulas, while Saha and Kayal~\textcolor{blue}{(2025)} recently introduced a class of Rényi  copula inaccuracy measures. Their semiparametric estimators demonstrated strong performance in model selection tasks based on both simulated and real data.

In Section~2, we present the foundational concepts and preliminaries on copula theory. In Section~3, we define the MCCFI measure, a copula-based inaccuracy measure, and investigate its main theoretical properties. Section~4 introduces the MSCFI measure and examines its relationship with the MCCFI framework. In Section~5, we extend the analysis to co-copula and dual copula settings by proposing the MCoCFI and MDCFI measures and studying their associated properties.
\section{Preliminary Results of Copula Function}\label{sec2}

This section outlines the foundational concepts and essential results related to copula theory, which serve as the basis for our subsequent developments. While the main focus of this work pertains to multivariate random vectors with more than two dimensions, we confine our preliminary discussion to the bivariate setting for clarity and ease of exposition.

\begin{definition}[Bivariate Copula]
	A function $C: [0,1]^2 \rightarrow [0,1]$ is termed a bivariate copula if it satisfies the following conditions (see Nelson~(\textcolor{blue}{2006})):
	\begin{itemize}
		\item Boundary conditions: $C(u,0) = 0 = C(0,v)$, and $C(u,1) = u$, $C(1,v) = v$ for all $u, v \in [0,1]$;
		\item 2-increasing: For all $u_1 \le u_2$ and $v_1 \le v_2$ in $[0,1]$, it holds that
		\[
		C(u_2,v_2) - C(u_2,v_1) - C(u_1,v_2) + C(u_1,v_1) \geq 0.
		\]
	\end{itemize}
\end{definition}

The copula function is bounded from below and above by the well-known Fr\'echet–Hoeffding bounds, defined as:
\begin{align}\label{eqn2.1}
	\max\{u + v - 1, 0\} \leq C(u,v) \leq \min\{u,v\}, \quad \text{for all } u, v \in [0,1].
\end{align}

\begin{theorem}[Sklar's Theorem]
	Let $\mathbb F$ be the joint cumulative distribution function (CDF) of a random vector $(X_1,X_2,...,X_n)$ with marginal CDFs $F_1,F_2,...,F_n$. Then, there exists a copula function $C$ such that
	\begin{align}\label{eq:sklar}
		\mathbb F(z_1,z_2,...,z_n) = C(F_1(z_1), F_2(z_2),...,F_n(z_n)), \quad \text{for all } z_1, z_2,...,z_n \in \mathbb{R}.
	\end{align}
\end{theorem}

\begin{definition}[Survival Copula]
	Let $\overline{F}$ denote the survival function of $(X_1,X_2,...,X_n)$, and let $\overline{F}_1, \overline{F}_2,...,\overline{F}_n$ be the marginal survival functions. Then, the survival copula $\overline{C}$ satisfies
	\begin{align}\label{eq:survival}
		\overline{\mathbb F}(z_1,z_2) = \overline{C}(\overline{F}_1(z_1) \overline{F}_2(z_2),...,\overline{F}_n(z_n)).
	\end{align} and in bivariate case
	$$	\overline{F}(z_1,z_2) = \overline{C}(\overline{F}_1(z_1), \overline{F}_2(z_2)).$$
\end{definition}

The connection between the survival copula and the standard copula is given by
\begin{align}
	\overline{C}(u,v) = u + v - 1 + C(1 - u, 1 - v).
\end{align}

Additionally, related constructs such as the co-copula and dual copula are defined as follows:
\begin{align}
	\widehat{C}(u,v) &= P(X > z_1 \text{ or } X > z_2) = 1 - C(1 - u, 1 - v) = u + v - \overline{C}(u,v), \\
	\widetilde{C}(u,v) &= P(X < z_1 \text{ or } X < z_2) = 1 - \overline{C}(1 - u, 1 - v) = u + v - C(u,v).
\end{align}

\begin{definition}[Radial Symmetry]
	Let $\textbf{X} = (Z_1,Z_2)$ be a bivariate random vector and $(a,b) \in \mathbb{R}^2$. We say that $\textbf{X}$ is radially symmetric about $(a,b)$ if the vectors $(Z_1 - a, Z_2 - b)$ and $(a - Z_1, b - Z_2)$ share the same joint distribution. Equivalently, the joint CDF $F$ and survival function $\overline{F}$ satisfy
	\begin{align}
		F(a + z_1, b + z_2) = \overline{F}(a - z_1, b - z_2), \quad \text{for all } (z_1, z_2) \in \mathbb{R}^2.
	\end{align}
	See Amblard~(\textcolor{blue}{2002}) for further details.
\end{definition}

\begin{definition}[Orthant Stochastic Orders]
	Let $\textbf{X} = (X_1, \ldots, X_n)$ and $\textbf{Y} = (Y_1, \ldots, Y_n)$ be two random vectors with joint CDFs $F$ and $G$, and survival functions $\overline{F}$ and $\overline{G}$, respectively. Then (see Shaked~(\textcolor{blue}{2007})):
	\begin{itemize}
		\item $\textbf{X} \le_{UO} \textbf{Y}$ (upper orthant order) if $\overline{F}(z_1,\ldots,z_n) \le \overline{G}(z_1,\ldots,z_n)$ for all $z_i \in \mathbb{R}$;
		\item $\textbf{X} \le_{LO} \textbf{Y}$ (lower orthant order) if $F(z_1,\ldots,z_n) \ge G(z_1,\ldots,z_n)$ for all $z_i \in \mathbb{R}$.
	\end{itemize}
\end{definition}
Symmetry plays a pivotal role in statistical modeling, offering simplification and deeper interpretation through structural invariance under transformations such as rotation, reflection, or inversion. While several forms of symmetry exist in the bivariate case, namely, marginal, exchangeable, joint, and radial symmetry. We restrict our focus to radial symmetry due to its relevance in the modeling of balanced systems.
For comprehensive coverage of copula theory and its properties, we refer the reader to Nelson~(\textcolor{blue}{2006}).
\section{Cumulative Copula Fractional Inaccuracy}\label{sec3}
In this section, we introduce a novel Multivariate Cumulative Copula Fractional Inaccuracy (MCCFI) measure based on multivariate copula functions and investigate its key properties. The proposed measure serves as a robust tool for quantifying inaccuracies in multivariate experimental outcomes, particularly in scenarios where the true underlying dependence structure is misrepresented by an assumed reference copula.
\begin{definition}
Let $\mathbb{X}=(X_1,\dots,X_n)$ and $\mathbb{Y}=(Y_1,\dots,Y_n)$ represent $n$-dimensional RVs with their respective cumulative copula functions (CCFs) denoted by $C_\mathbb{X}$ and $C_\mathbb{Y}$, respectively. Let $F_i(\cdot)$ and $G_i(\cdot)$ denote the univariate cumulative distribution functions (CDFs) of $X_i$ and $Y_i$, respectively, for $i=1,\dots,n\in\mathbb{N}$.   The MCCFI measure between $\mathbb{X}$ and $\mathbb{Y}$ is
\begin{align}\label{eq3.1}
CCFI(\mathbb{X},\mathbb{Y})&=\int_{\mathbb I^n}C_\mathbb{X}(\textbf{v})\Big\{-\mathrm{Ln}_\eta \left(C_\mathbb{Y}\big(\mathbb{G}({\mathbb F}^{-1}(\textbf{v})\big)\right)\Big\}^{\frac{1}{\eta}}d\textbf{v},
\end{align}
\end{definition}
where $\mathbb I^n=\underbrace{[0,1]\times[0,1]\cdots[0,1]}_{n-times}$ $\textbf{v}=(v_1,\cdots,v_n)$ and $\mathbb{G}\big(\mathbb{F}^{-1}(\textbf{v})\big)=\big(G_1\big(F_1^{-1}(v_1)\big),\cdots,G_n\big(F_n^{-1}(v_n)\big)\big)$.
Using the relation $\mathrm{Ln}_p x\approx \log x^{p!}, 0<p<1,$ (see p. 125 of Jumarie (\textcolor{blue}{2012})). Thus, we get an approximation of (3.1), which is given as 
\begin{align}\label{eq2.1}
	CCFI(\mathbb{X},\mathbb{Y})&\approx (\eta!)^{\frac{1}{\eta}}\int_{\mathbb I^n}C_\mathbb{X}(\textbf{v})\Big\{-\log \left(C_\mathbb{Y}\big(\mathbb{G}({\mathbb F}^{-1}(\textbf{v})\big)\right)\Big\}^{\frac{1}{\eta}}d\textbf{v},
\end{align}
Similar analogy to Foroghi et al. (\textcolor{blue}{2022}, Remark 2.1), a modified version of MCCFI is be defined as 
\begin{align}\label{eq2.1}
	\widetilde{CCFI}(\mathbb{X},\mathbb{Y})= -\int_{\mathbb I^n}C_\mathbb{X}(\textbf{v})\Big\{\mathrm{Ln}_\eta \left(C_\mathbb{Y}\big(\mathbb{G}({\mathbb F}^{-1}(\textbf{v})\big)\right)\Big\}d\textbf{v}\approx \eta! CCI(\mathbb X,\mathbb Y),
\end{align}
where $CCI(\mathbb X,\mathbb Y)$ is cumulative copula inaccuracy (see Hosseini, \textcolor{blue}{2021}), given by 
\begin{align}\label{eq1.25}
	CCI((\mathbb{X},\mathbb{Y}))=-\int_{0}^{1}\int_{0}^{1}C_\mathbb{X}(u,v)\log C_\mathbb{Y}(G_1(F_1^{-1}(u)),G_2(F_2^{-1}(v))) dvdu.
\end{align}
 Here $C_{\mathbb X}$ is assumed as true copula and $C_{\mathbb Y}$ as reference copula.
In particular, when $X_i\overset{\mathrm{st}}{=}Y_i$, the MCCFI measure in (\ref{eq2.1}) reduces to
\begin{align}\label{eq2.2}
CCFI(\mathbb{X},\mathbb{Y})=\int_{\mathbb I^n}C_\mathbb{X}(\textbf{v})\big\{-\mathrm{Ln}_\eta C_\mathbb{Y}(\textbf{v})\big\}^{\frac{1}{\eta}}d\textbf{v},\delta>0.
\end{align} 
Further, the CCFI measure given by (\ref{eq2.2}) becomes multivariate  cumulative copula  fractional entropy (MCCFE), when $\mathbb{X}$ and $\mathbb{Y}$ are identically distributed. It is given by
\begin{align}\label{eq2.3}
CCFE(\mathbb{X})=\int_{\mathbb I}C_{\mathbb X}(\textbf u)\big\{-\log C_\mathbb{X}(\textbf{v})\big\}^{\frac{1}{\eta}}d\textbf{v},\eta>0.
\end{align} 
 This is not hard to see that MCCFE is always positive. 
 
\begin{example}\label{ex3.1}
Suppose $\textbf{X}=(X_1,X_2)$ and $\textbf{Y}=(Y_1,Y_2)$ are associated with Gumbel and FGM copula functions
\begin{align*}
C_{\textbf{X}}(u,v)=\exp\left[-\left((-\log u)^\theta+\left((-\log v)\right)^\theta\right)^{\frac{1}{\theta}}\right]{and}~
C_{\textbf{Y}}(u,v)=uv\left[1+\theta(1-u)(1-v)\right],
\end{align*}
respectively. Further, assume that $X_1$ and $X_2$ are two standard exponential rvs and $Y_1$ and $Y_2$ are two exponential rvs with parameters $\mu_1$ and $\mu_2$, respectively. Thus, the MCCFI measure for $0<\gamma\ne1$ in (\ref{eq2.1}) is obtained as
\begin{align}\label{eq3.7*}
CCFI(\textbf{X},\textbf{Y})&=(\eta!)^{\frac{1}{\eta}}\int_0^1\int_0^1\exp\left[-\left((-\log u)^\theta+\left(-\log v\right)^\theta\right)^{\frac{1}{\theta}}\right]\nonumber\\
&\times\left[-\log\left\{\left(1-(1-u)^{\mu_1}\right)\left(1-(1-v)^{\mu_2}\right)\left(1+\theta\left(1-u\right)^{\mu_1}\left(1-v\right)^{\mu_2}\right)\right\}\right]^{\frac{1}{\eta}} dvdu.
\end{align}
Note that the inaccuracy measure in Equation~(\ref{eq3.7*}) is difficult to evaluate explicitly. Therefore, we have illustrated it in Figure~\ref{fig1} using numerical methods for various combinations of parameters.
					\begin{figure}[H] 
	\centering
\includegraphics{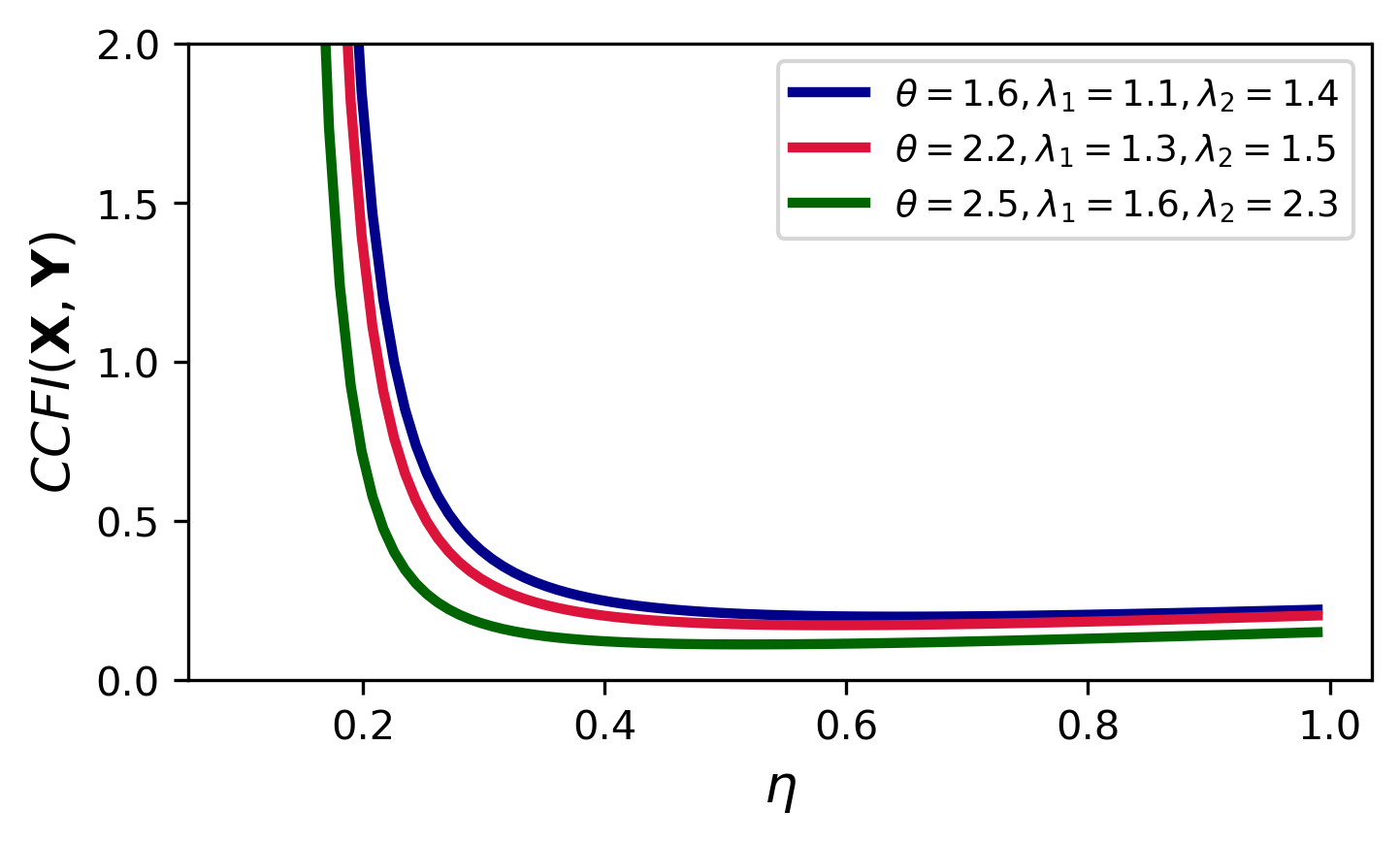}
\caption{Plot of CCFI for Gumbel and FGM copula distributions.}\label{fig1}
\end{figure}
\end{example}
Now, we proceed to derive bounds for the MCCFI measure by employing the Fréchet–Hoeffding bounds for copulas, as given in equation (\ref{eqn2.1}), under the assumption of a proportional reversed hazard rate (PRHR) model. Let $G_1(\cdot)$ and $F_1(\cdot)$ be two cumulative distribution functions (CDFs). These functions are said to satisfy the PRHR model if $G_1(t) = F_1^{\gamma}(t)$ for some $\gamma > 0$ and $t > 0$, $t \neq 1$. For a comprehensive discussion of the PRHR model, we refer the reader to Gupta (\textcolor{blue}{2007}).
\begin{proposition}\label{prop3.1}
Suppose \(\mathbf{X} = (X_1, X_2)\) and \(\mathbf{Y} = (Y_1, Y_2)\) are two bivariate RVs. Assume that \(Y_1\) and \(Y_2\) are independent. Let \(F_i(\cdot)\) and \(G_i(\cdot)\) denote the CDFs of \(X_i\) and \(Y_i\), respectively, for \(i = 1, 2\). Suppose further that  $
G_1(t) = F_1^\gamma(t) \quad \text{and} \quad G_2(t) = F_2^\delta(t) $
for \(t > 0\), where \(\gamma, \delta \ (\neq 1) \in \mathbb{R}^+\). Then,
\begin{align*}
\zeta(\eta,\gamma,\delta)	\leq CCFI(\textbf X,\textbf Y)\leq \xi(\eta,\gamma,\delta),
\end{align*}
where 
\begin{eqnarray*}
\xi(\eta,\gamma,\delta)&=&\left(\eta!\right)^{\frac{1}{\eta}} \cdot \frac{\Gamma\left(1 + \frac{1}{\eta}\right)}{2 \cdot 3^{1 + \frac{1}{\eta}}} (\gamma + \delta).
\end{eqnarray*}
and 
\begin{align*}
	\zeta(\eta,\gamma,\delta)=
	(\eta!)^{\frac{1}{\eta}} \Bigg[
	\gamma \cdot \left(
	\frac{3^{-1/\eta} \, \Gamma\left(\frac{1}{\eta}\right)}{6\eta}
	+ \Gamma\left(\frac{\eta + 1}{\eta}\right) \left(2^{\frac{-\eta - 1}{\eta}} - 3^{\frac{-\eta - 1}{\eta}} \right)
	\right)\\
	~~~+ \delta \cdot \left(
	\frac{\Gamma\left(\frac{2\eta + 1}{\eta}\right) \cdot {}_2F_1\left(1, \frac{2\eta + 1}{\eta}; 2; -\frac{1}{2} \right)}{4 \cdot 2^{1/\eta}}
	+ \frac{\Gamma\left(1 + \frac{1}{\eta}\right)}{2 \cdot 3^{1 + \frac{1}{\eta}}}
	\right)
	\Bigg].
\end{align*}
\end{proposition}

\begin{proof}
Denote $\textbf{G}\big(\textbf{F}^{-1}(\textbf{v})\big)=\big(G_1\big(F_1^{-1}(v_1)\big),G_2\big(F_2^{-1}(v_2)\big)\big)$. Using Fr\'echet-Hoeffding bounds, we obtain 
\begin{align}\label{eq3.6**}
\max\{v_1+v_2-1,0\}\{-\mathrm{Ln}_\eta C_{\textbf{Y}}(\textbf{G}\big(\textbf{F}^{-1}(\textbf{v})\big))\}^{\frac{1}{\eta}}&\le C_{\textbf{X}}(v_1,v_2)\{-\mathrm{Ln}_\eta C_{\textbf{Y}}(\textbf{G}\big(\textbf{F}^{-1}(\textbf{v})\big))\}^{\frac{1}{\eta}}\nonumber\\
&\le \min\{v_1,v_2\}\{-\mathrm{Ln}_\eta C_{\textbf{Y}}(\textbf{G}\big(\textbf{F}^{-1}(\textbf{v})\big))\}^{\frac{1}{\eta}}.
\end{align}
 Integrating (\ref{eq3.6**}) with respect to $v_1$ and $v_2$ in the region $[0,1]\times[0,1]$, we get
\begin{eqnarray}\label{eq3.10}
I_1\ge CCFI(\textbf{X},\textbf{Y}) \ge I_2,
\end{eqnarray}
where 
\begin{align}
I_1&= \int_{0}^{1}\int_{0}^{1}\max\{v_1+v_2-1,0\}\{-\mathrm{Ln}_\eta C_{\textbf{Y}}(G_1(F_1^{-1}(v_1)),G_2(F_2^{-1}(v_2)))\}^{\frac{1}{\eta}}dv_2dv_1\nonumber\\
&=\int_{0}^{1}\int_{0}^{1}\max\{v_1+v_2-1,0\}\left(-\mathrm{Ln}_\eta(v_1^\gamma v_2^\delta)\right)^{\frac{1}{\eta}}dv_2dv_1 ~(\mbox{from~independence})\nonumber\\
&=\int_{0}^{1}\int_{0}^{1}\max\{v_1+v_2-1,0\}\left(-\mathrm{Ln}_\eta(v_1^\gamma)\right)^{\frac{1}{\eta}}dv_2dv_1\nonumber\\ &~~~+\int_{0}^{1}\int_{0}^{1}\max\{v_1+v_2-1,0\}\left(-\mathrm{Ln}_\eta( v_2^\delta)\right)^{\frac{1}{\eta}}dv_2dv_1\nonumber\\
&\approx \gamma (\eta!)^{\frac{1}{\eta}}\Big(\int_{0}^{1}\int_{0}^{1}\max\{v_1+v_2-1,0\}\left(-\log(v_1)\right)^{\frac{1}{\eta}}dv_2dv_1\nonumber\\ &~~~+ \delta (\eta!)^{\frac{1}{\eta}}\int_{0}^{1}\int_{0}^{1}\max\{v_1+v_2-1,0\}\left(-\log( v_2)\right)^{\frac{1}{\eta}}dv_2dv_1\Big)\nonumber\\
&= \gamma (\eta!)^{\frac{1}{\eta}}\Big(\int_{0}^{1}\int_{1-v_1}^{1}(v_1+v_2-1)\left(-\log(v_1)\right)^{\frac{1}{\eta}}dv_2dv_1\nonumber\\ &~~~+ \delta (\eta!)^{\frac{1}{\eta}}\int_{0}^{1}\int_{1-v_1}^{1}(v_1+v_2-1)\left(-\log( v_2)\right)^{\frac{1}{\eta}}dv_2dv_1\Big)\nonumber\\
&=\gamma (\eta!)^{\frac{1}{\eta}}\frac{3^{-1-\frac{1}{\eta}}\Gamma\left(1+\frac{1}{\eta}\right)}{2}+\delta(\eta!)^{\frac{1}{\eta}}\frac{\Gamma\left(1 + \frac{1}{\eta}\right)}{2 \cdot 3^{1 + 1/\eta}}=\xi(\eta,\gamma,\delta)
\end{align}
and 
\begin{align*}
I_2&= \int_{0}^{1}\int_{0}^{1}\min\{v_1,v_2\}\{-\mathrm{Ln}_\eta C_{\textbf{Y}}(G_1(F_1^{-1}(v_1)),G_2(F_2^{-1}(v_2)))\}^{\frac{1}{\eta}}dv_2dv_1\nonumber\\
&=\int_{0}^{1}\int_{0}^{1} \min\{v_1,v_2\}(-\mathrm{Ln}_\eta(v_1^\gamma v_2^\delta))^{\frac{1}{\eta}}dv_2dv_1 ~(\mbox{from~independence})\nonumber\\
&=\bigg(\int_{0}^{1}\int_{0}^{v_1}v_2(-\mathrm{Ln}_\eta(v_1^\gamma v_2^\delta))^{\frac{1}{\eta}}dv_2dv_1+\int_{0}^{1}\int_{v_1}^{1}v_1(-\mathrm{Ln}_\eta(v_1^\gamma v_2^\delta))^{\frac{1}{\eta}}dv_2dv_1\bigg)\nonumber\\
&=\gamma(\eta!)^{\frac{1}{\eta}}\int_0^1\int_0^{v_1}v_2(-\log v_1)^{\frac{1}{\eta}}dv_2dv_1+\delta(\eta!)^{\frac{1}{\eta}}\int_0^1\int_0^{v_1}v_2(-\log v_2)^{\frac{1}{\eta}}dv_2dv_1\nonumber\\
&~~~+\gamma(\eta!)^{\frac{1}{\eta}}\int_0^1\int_{v_1}^1 v_1(-\log v_1)^{\frac{1}{\eta}}dv_2dv_1+\delta(\eta!)^{\frac{1}{\eta}}\int_0^1\int_{v_1}^1 v_1(-\log v_2)^{\frac{1}{\eta}}dv_2dv_1\nonumber\\
&=\gamma(\eta!)^{\frac{1}{\eta}}\frac{3^{-\frac{1}{\eta}}\Gamma(\frac{1}{\eta})}{6\eta}+\delta(\eta!)^{\frac{1}{\eta}}
\frac{\Gamma\left(\frac{2\eta + 1}{\eta}\right) \cdot {}_{2}F_{1}\left(1, \frac{2\eta + 1}{\eta}; 2; -\frac{1}{2} \right)}{4 \cdot 2^{1/\eta}}\nonumber\\
&~~+\gamma(\eta!)^{\frac{1}{\eta}}\Gamma\left(\frac{\eta+1}{\eta}\right)\left(2^{\frac{-\eta-1}{\eta}}-3^{\frac{-\eta-1}{\eta}}\right)+\delta(\eta!)^{\frac{1}{\eta}}\frac{\Gamma\left(1+\frac{1}{\eta}\right)}{2\cdot3^{1+\frac{1}{\eta}}},
\end{align*}
where ${}_{2}F_{1}(\cdot,\cdot,\cdot)$ is the hypergeometric function.
This completes the proof of the first part. 
\end{proof}
 Next, we discuss comparison study between two MCCFI measures. The comparison of two multi-dimensional inaccuracy measures are needed to understand the insights of the complex interactions and dependencies in multi-variable complex systems. It helps to select suitable analytical tools and optimizing models. Note that in machine learning or system analysis, comparing information measures is essential for validating how well models capture dependencies or reduce uncertainty.
\begin{proposition}\label{prop3.2}
Suppose \(\mathbb{X} = (X_1, \dots, X_n)\), \(\mathbb{Y} = (Y_1, \dots, Y_n)\), and \(\mathbb{Z} = (Z_1, \dots, Z_n)\) are \(n\)-dimensional RVs with copulas \(C_{\mathbb{X}}\), \(C_{\mathbb{Y}}\), and \(C_{\mathbb{Z}}\), respectively. Assume that the CDFs of \(X_i\), \(Y_i\), and \(Z_i\) are \(F_i\), \(G_i\), and \(H_i\), respectively, where  
$
H_i(x_i) = F_i^{\gamma_i}(x_i)~\text{and}~C_{\mathbb{X}} = C_{\mathbb{Z}}, \quad i = 1, \dots, n.
$
Then, for \(0 < \eta < 1\), we have  
\begin{align*}
	\text{CCFI}(\mathbb{Z}, \mathbb{Y}) &\ge \text{CCFI}(\mathbb{X}, \mathbb{Y})~~for~~\gamma_i>1\\
	  	\text{CCFI}(\mathbb{Z}, \mathbb{Y}) &\le \text{CCFI}(\mathbb{X}, \mathbb{Y})~~for~~0<\gamma_i<1; i=1,2,\cdots,n.
\end{align*}  
\end{proposition}
\begin{proof}
 From (\ref{eq3.1}), we have
\begin{align}\label{eq3.12}
CCFI(\mathbb{Z},\mathbb{Y})=\int_{\mathbb I}C_\mathbb{Z}(\textbf{v})\Big\{-\mathrm{Ln}_\eta C_\mathbb{Y}\big(G_1(H_1^{-1}(v_1)),\cdots,G_n(H_n^{-1}(v_n))\big)\Big\}^{\frac{1}{\eta}}d\textbf{v}.
\end{align}
Under the assumption $C_{\mathbb{X}}=C_{\mathbb{Z}}$, and then changing the variables $u_i=p_i^{\gamma_i}$ in (\ref{eq3.12}), we obtain
\begin{align}\label{eq3.13}
CCFI(\mathbb{Z},\mathbb{Y})=&\int_{\mathbb I}\Big(\prod_{i=1}^{n}\gamma_ip_i^{\gamma_i-1}\Big) C_\textbf{X}(p_1^{\gamma_1},\cdots,p_n^{\gamma_n})\nonumber\\
&\times\Big\{-\mathrm{Ln}_\eta C_\mathbb{Y}\big(G_1(H_1^{-1}(p^{\gamma_1}_1)),\cdots,G_n(H_n^{-1}(p^{\gamma_n}_n))\big)\Big\}^{\frac{1}{\eta}}dp_1\cdots dp_n.
\end{align}
Further, according to the assumption, we have $H_i(x_i)=F_i^{\gamma_i}(x_i),~i=1,\cdots,n$. Thus, for each $w_i$ belonging to the interval $(0,1)$, we get 
\begin{align}\label{eq3.14}
H^{-1}_i(w_i)=F_i^{-1}\Big(w_i^{\frac{1}{\gamma_i}}\Big),~~ i=1,\cdots,n.
\end{align}
Now, using (\ref{eq3.14}) in (\ref{eq3.13}), we obtain
\begin{align}\label{eq3.15}
CCFI(\mathbb{Z},\mathbb{Y})=&\int_{\mathbb I}\Big(\prod_{i=1}^{n}\gamma_ip_i^{\gamma_i-1}\Big) C_\textbf{X}(p_1^{\gamma_1},\cdots,p_n^{\gamma_n})\nonumber\\
&\times\Big\{-\mathrm{Ln}_\eta C_\textbf{Y}\big(G_1(F_1^{-1}(p_1)),\cdots,G_n(F_n^{-1}(p_n))\big)\Big\}^{\frac{1}{\eta}}dp_1\cdots dp_n.
\end{align}
Now, we consider the following two cases:
\\

{\bf Case-I:} Let $\gamma_i>1$. Then, $\prod_{i=1}^{n}p_i^{\gamma_i-1}<1$ and $C_X(p_1,\cdots,p_n)\ge C_X(p^{\gamma_1}_1,\cdots,p^{\gamma_n}_n)$ since copula is an increasing function. Now, using these arguments in (\ref{eq3.15}), after some algebra we get
\begin{align*}
CCFI(\mathbb{Z},\mathbb{Y})\le \prod_{i=1}^n \gamma_i~ CCFI(\mathbb{X},\mathbb{Y}).
\end{align*}
{\bf Case-II:} Let $\gamma_i\in(0,1)$. Thus, $\prod_{i=1}^{n}p_i^{\gamma_i-1}>1$ and $C_X(p_1,\cdots,p_n)\le C_X(p^{\gamma_1}_1,\cdots,p^{\gamma_n}_n)$. Using these in (\ref{eq3.15}), after simplification we obtain
\begin{align*}
CCFI(\mathbb{Z},\mathbb{Y})\ge \prod_{i=1}^n \gamma_i CCFI(\mathbb{X},\mathbb{Y}).
\end{align*}
\end{proof}
\begin{proposition}\label{prop3.3}
For the RVs \(\mathbb{X}\), \(\mathbb{Y}\), and \(\mathbb{Z}\), let their copulas be denoted by \(C_{\mathbb{X}}\), \(C_{\mathbb{Y}}\), and \(C_{\mathbb{Z}}\), respectively. Assume that \(F_i\), \(G_i\), and \(H_i\) are the CDFs of \(X_i\), \(Y_i\), and \(Z_i\), respectively. Further, suppose the following conditions hold $
C_{\mathbf{Z}} = C_{\mathbf{Y}}, \quad G_i = F_i^{\gamma_i}, \quad \text{and} \quad H_i = F_i^{\delta_i},$
where \(\gamma_i < \delta_i\) for all \(i = 1, \dots, n\) and \(n \in \mathbb{N}\).  
Then, $
\text{CCFI}(\mathbb{X}, \mathbb{Y}) \le \text{CCFI}(\mathbb{X}, \mathbb{Z}).$
\end{proposition}
\begin{proof}
Since $G_i=F_i^{\gamma_i}$ and $H_i=F_i^{\delta_i},$ for $i=1,\cdots,n\in\mathbb{N}$, we have
\begin{align}\label{eq3.16}
G_i(F_i^{-1}(\omega_i))=\omega_i^{\gamma_i}~~\text{and}~~H_i(F_i^{-1}(\omega_i))=\omega_i^{\delta_i},~~\omega_i\in[0,1].
\end{align}
Further, $u_i^{\gamma_i}\ge u_i^{\delta_i}$, since  $u_i\in[0,1]$ and $\gamma_i<\delta_i,~i=1,\cdots,n\in\mathbb{N}.$ Therefore,
\begin{align}\label{eq3.17}
C_\textbf{Y}(v_1^{\gamma_1},\cdots,v_n^{\gamma_n})\ge C_\textbf{Y}(v_1^{\delta_1},\cdots,v_n^{\delta_n}).
\end{align}
Using (\ref{eq3.16}) and (\ref{eq3.17}), we obtain
\begin{align*}
C_\textbf{Y}\Big(G_1(F_1^{-1}(v_1)),\cdots,G_n(F_n^{-1}(v_n))\Big)&=C_\textbf{Y}(v_1^{\gamma_1},\cdots,v_n^{\gamma_n})\nonumber\\
&\ge C_\textbf{Y}(v_1^{\delta_1},\cdots,v_n^{\delta_n})\nonumber\\
&=C_\textbf{Z}\Big(H_1(F_1^{-1}(v_1)),\cdots,G_n(H_n^{-1}(v_n))\Big).
\end{align*}
Equivalently,
\begin{align*}
\Big\{-\mathrm{Ln}_{\eta} C_\textbf{Y}\Big(G_1(F_1^{-1}(v_1)),\cdots,G_n(F_n^{-1}(v_n))\Big)\Big\}^{\frac{1}{\eta}}\le \Big\{-\mathrm{Ln}_\eta C_\textbf{Z}\Big(H_1(F_1^{-1}(v_1)),\cdots,G_n(H_n^{-1}(v_n))\Big)\Big\}^{\frac{1}{\eta}}
\end{align*}
It follows that
\begin{align*}
\int_{0}^{1}\cdots \int_{0}^{1} C_X(v_1,\cdots,v_n)\Big\{-\mathrm{Ln}_\eta C_\textbf{Y}\Big(G_1(F_1^{-1}(v_1)),\cdots,G_n(F_n^{-1}(v_n))\Big)\Big\}^{\frac{1}{\eta}}dv_1\cdots dv_n\\
\le \int_{\mathbb I}C_{\mathbb X}(v_1,\cdots,v_n)\Big\{-\mathrm{Ln}_\eta C_\textbf{Z}\Big(G_1(F_1^{-1}(v_1)),\cdots,G_n(F_n^{-1}(v_n))\Big)\Big\}^{\frac{1}{\eta}}dv_1\cdots dv_n.\\
\end{align*} 
This completes the proof.
\end{proof}
\begin{proposition}\label{prop3.4}
If $\mathbb{X}\le_{LO}\mathbb{Y}$, then, for $0<\eta<1$, $CCFI(\mathbb{Z},\mathbb{X})\le CCFI(\mathbb{Z},\mathbb{Y})$.
\end{proposition}
\begin{proof}
We know that $\mathbb{X}\le_{LO}\mathbb{Y}$ implies $F(z_1,\cdots,z_n)\ge G(z_1,\cdots,z_n).$ According to the Sklar's Theorem, we have
\begin{align}\label{eq3.18}
C_\mathbb{X}\big(v_1,\cdots, v_n)\big)\ge C_\mathbb{Y}\big(v_1,\cdots, v_n\big).
\end{align}
After the transformation $z_i=H^{-1}_i(v_i)$, for $i=1,\cdots,n\in\mathbb{N}$ in (\ref{eq3.18}), we obtain
\begin{align}\label{eq3.19}
C_\mathbb{X}\big(F_1(H_1^{-1}(v_1)),\cdots, F_n(H_n^{-1}(v_n))\big)\ge C_\mathbb{Y}\big(G_1(H_1^{-1})(v_1)),\cdots, G_n(H_n^{-1}(v_n))\big).
\end{align}
 We obtain from (\ref{eq3.19}) as 
\begin{align*}
\Big\{-\mathrm{Ln}_\eta C_\mathbb{X}\big(F_1(H^{-1}_1(v_1)),\cdots, F_n(H^{-1}_n(v_n))\big)\Big\}^{{\frac{1}{\eta}}}\le \Big\{-\mathrm{Ln}_\eta C_\mathbb{Y}\big(G_1(H^{-1}_1(v_1)),\cdots, G_n(H^{-1}_n(v_n))\big)\Big\}^{\frac{1}{\eta}}\\
\end{align*}
It follows that 
\begin{align*}
	 \int_{\mathbb I}C_\mathbb{Z}(v_1,\cdots,v_n)\Big\{-\mathrm{Ln}_\eta C_\mathbb{X}\big(F_1(H^{-1}_1(v_1)),\cdots, F_n(H^{-1}_n(v_n))\big)\Big\}^{{\frac{1}{\eta}}}dv_1\cdots dv_n\\
	\le \int_{\mathbb I}C_\mathbb{Z}(v_1,\cdots,v_n)\Big\{-\mathrm{Ln}_\eta C_\mathbb{Y}\big(G_1(H^{-1}_1(v_1)),\cdots, G_n(H^{-1}_n(v_n))\big)\Big\}^{{\frac{1}{\eta}}}dv_1\cdots dv_n.
\end{align*}
\end{proof}
\begin{proposition}\label{prop3.5}
Let $\mathbb{X}\le_{LO}\mathbb{Y}$. Further, let $X_i\overset{\mathrm{st}}{=}Y_i$,~$i=1,\cdots,n$. Then,
 for $0<\eta<1$, $CCFI(\mathbb{X},\mathbb{Z})\ge CCFI(\mathbb{Y},\mathbb{Z})$;
\end{proposition}

\begin{proof}
Under the assumption $\mathbb{X}\le_{LO}\mathbb{Y}$, from (\ref{eq3.18}) we have
\begin{align}\label{eq3.20}
C_\mathbb{X}\big(F_1(z_1),\cdots, F_n(z_n)\big)\ge C_\mathbb{Y}\big(G_1(z_1),\cdots, G_n(z_n)\big).
\end{align}
Further, $X_i\overset{\mathrm{st}}{=}Y_i$,~$i=1,\cdots,n$. Thus, from (\ref{eq3.20}) we obtain
\begin{align}\label{eq3.21}
C_\mathbb{X}\big(F_1(z_1),\cdots, F_nz_n\big)\ge C_\mathbb{Y}\big(F_1(z_1),\cdots, F_nz_n\big).
\end{align}
Applying transformation $u_i=F_i(x_i)$ in (\ref{eq3.21}), we get
\begin{align}\label{eq3.22}
C_\mathbb{X}\big(v_1,\cdots, v_n\big)\ge C_\mathbb{Y}\big(v_1,\cdots, v_n\big).
\end{align}
From (\ref{eq3.22}), we get
\begin{align}\label{3.23}
&C_\mathbb{X}\big(v_1,\cdots, v_n\big)\Big\{-\mathrm{Ln}_\eta C_Z\big(H_1(F_1^{-1}(v_1)),\cdots,H_n(F_n^{-1}(v_n))\big)\Big\}^{\frac{1}{\eta}}\nonumber\\
&\ge C_\mathbb{Y}\big(v_1,\cdots, v_n\big)\Big\{-\mathrm{Ln}_\eta C_Z\big(H_1(F_1^{-1}(v_1)),\cdots,H_n(F_n^{-1}(v_n))\big)\Big\}^{\frac{1}{\eta}}.
\end{align}
Now, integrating both sides of (\ref{3.23}) over the region $[0,1] \times [0,1] \times \cdots \times [0,1]$ yields the desired result. This concludes the proof.
\end{proof}

\begin{proposition}\label{prop3.6}
Suppose that $\mathbb{X},~\mathbb{Y}$ and $\mathbb{Z}$ have copulas $C_\mathbb{X},~C_\mathbb{Y}$ and $C_\mathbb{Z}$, respectively. 
\begin{enumerate}
\item If $\mathbb{Z}\le_{LO}\mathbb{Y},~\mathbb{Z}\le_{LO}\mathbb{X}$ and $Z_i\overset{\mathrm{st}}{=}X_i$,~$i=1,\cdots,n$. 
Then, $CCFI(\mathbb X,\mathbb Z)\leq CCFI(\mathbb X,\mathbb Y)\le CCFI(\mathbb Z,\mathbb Y).$;
\item  If $\mathbb{X}\le_{LO}\mathbb{Z}\le_{LO}\mathbb{Y}$ and $Z_i\overset{\mathrm{st}}{=}X_i$,~$i=1,\cdots,n$. \\Then,
 $CCFI(\mathbb{X},\mathbb{Y})\ge \max\{CCFI(\mathbb{Z},\mathbb{Y})\le CCFI(\mathbb{X},\mathbb{Z})\}$;
\item If $\mathbb{Y}\le_{LO}\mathbb{Z},~\mathbb{X}\le_{LO}\mathbb{Z}$ and$Z_i\overset{\mathrm{st}}{=}X_i$,~$i=1,\cdots,n$. \\Then,
 $CCFI(\mathbb Z,\mathbb Y)\leq CCFI(\mathbb{X},\mathbb{Y})\ge CCFI(\mathbb X,\mathbb Z)$.
\end{enumerate}
\end{proposition}

\begin{proof}
We provide the proof for part $(1)$ only; the remaining parts can be established in a similar manner.
 Since $\mathbb{Z}\le_{LO}\mathbb{X}$ and $Z_i\overset{\mathrm{st}}{=}X_i$,~$i=1,\cdots,n$, we have 
\begin{align}\label{eq3.24}
C_\textbf{Z}(v_1,\cdots,v_n)\ge C_\textbf{X}(v_1,\cdots,v_n),~~u_i\in[0,1].
\end{align}
 From (\ref{eq3.24}), we get
\begin{align*}
& C_\textbf{Z}(v_1,\cdots,v_n)\Big\{-\mathrm{Ln}_\eta C_\textbf{Y}\big(G_1(F^{-1}_1(v_1)),\cdots, G_n(F^{-1}_n(v_n))\big)\Big\}^{\frac{1}{\eta}}\nonumber\\
&\ge C_\textbf{X}(v_1,\cdots,v_n)\Big\{-\mathrm{Ln}_\eta C_\textbf{Y}\big(G_1(F^{-1}_1(v_1)),\cdots, G_n(F^{-1}_n(v_n))\big)\Big\}^{\frac{1}{\eta}}\nonumber\\
\end{align*}
Equivalently,
\begin{align*}
	\int_{\mathbb I}C_\textbf{Z}(v_1,\cdots,v_n)&\Big\{-\mathrm{Ln}_\eta C_\textbf{Y}\big(G_1(F^{-1}_1(v_1)),\cdots, G_n(F^{-1}_n(v_n))\big)\Big\}^{\frac{1}{\eta}}dv_1\cdots dv_n\nonumber\\
	&\ge \ \int_{\mathbb I}C_\textbf{X}(v_1,\cdots,v_n)\Big\{-\mathrm{Ln}_\eta C_\textbf{Y}\big(G_1(F^{-1}_1(v_1)),\cdots, G_n(F^{-1}_n(v_n))\big)\Big\}^{\frac{1}{\eta}}dv_1\cdots dv_n.\nonumber\\
\end{align*}
Thus, we obtain
\begin{align}\label{eq3.25}
	CCFI(\textbf{Z},\textbf{Y})\ge CCFI(\textbf{X},\textbf{Y}).
\end{align}
Further, let $\mathbf Z\leq_{LO}\mathbf Y$. Then
\begin{align*}
	C_{\mathbf Z}(H_1(z_1),...,H_n(z_n))\geq C_{\mathbf Y}(G_1(z_1),...,G_n(z_n))
\end{align*}
Therefore, 
\begin{align*}
	\int_{\mathbb I}C_{\mathbf X}(v_1,...,v_n)&\left(-\mathrm{Ln}_\eta\left(C_{\mathbf Z}(H_1(F_1^{-1}(u_1)),...,H_n(F_n^{-1}(u_n)))\right)\right)^{\frac{1}{\eta}}d\textbf v\leq \\	&\int_{\mathbb I}C_{\mathbf X}(v_1,...,v_n)\left(-\mathrm{Ln}_\eta\left(C_{\mathbf Y}(G_1(F_1^{-1}(u_1)),...,G_n(F_n^{-1}(u_n)))\right)\right)^{\frac{1}{\eta}}d\textbf v,
\end{align*}
which implies 
\begin{align}\label{eq3.26}
	CCFI(\mathbf X,\mathbf Z)\leq CCFI(\mathbf X,\mathbf Y)
\end{align}
After combining (\ref{eq3.25}) and (\ref{eq3.26}), we obtain
\begin{align*}
CCFI(\mathbf X,\mathbf Z)\leq CCFI(\textbf X,\textbf Y)\le CCFI(\mathbf Z,\mathbf Y).
\end{align*}
\end{proof}
\section{Multivariate Survival  Copula Fractional Inaccuracy Measure}\label{sec4}
In information theory, modeling the dependence between random variables is vital for analyzing multivariate systems. The survival copula plays a central role in describing the joint upper-tail behavior, particularly through probabilities like $P(X_1 > z_1, \ldots, X_n > z_n)$, which reflect the likelihood of extreme outcomes occurring simultaneously. This is especially important in systems with dependent components, where such joint exceedances indicate potential failure or critical events. In high-dimensional applications like sensor networks, the survival copula helps reveal how extreme readings co-occur across different sensors. To capture this behavior more effectively, we introduce a generalized inaccuracy measure based on the multi-dimensional survival copula.

\begin{definition}
The multivariate survival copula fractional  inaccuracy (MSCFI) measure between $\mathbb{X}$ and $\mathbb{Y}$, for $0<\eta<1$ is defined as 
\begin{align}\label{eq4.1}
SCFI(\mathbb{X},\mathbb{Y})=\int_{\mathbb I}\overline{C}_\mathbb{X}(\textbf{v})\{-\mathrm{Ln}_\eta \overline{C}_\mathbb{Y}(\overline{\mathbb{G}}(\mathbb{\mathbb{F}}^{-1}(\textbf{v})))\}^{\frac{1}{\eta}}d\textbf{v},
\end{align}
where $\textbf{v}=(v_1,\cdots,v_n)$ and $\overline{\mathbb{G}}(\overline{\mathbb{F}}^{-1}(\textbf{v}))=\Big(\overline{G}_1(\overline{F}^{-1}_1(v_1),\cdots,\overline{G}_n(\overline{F}^{-1}_n(v_n))\Big)$.
\end{definition}
In particular, when $X_i\overset{\mathrm{st}}{=}Y_i$, the MSCFI measure is expressed as 
\begin{align}\label{eq4.2}
SCFI(\mathbb{X},\mathbb{Y})=\int_{\mathbb I}\overline C_\mathbb{X}(\textbf{v})\big\{-\mathrm{Ln}_\eta \overline C_\mathbb{Y}(\textbf{v})\big\}^{\frac{1}{\eta}}d\textbf{v},~~0<\eta<1.
\end{align} 
The MSCFI measure in (\ref{eq4.2}) becomes multivariate fractional entropy based on survival copula for $\overline C_\mathbb{X}=\overline C_\mathbb{Y}$. The multivariate survival copula fractional entropy (MSCFE)  is given by 
\begin{align}\label{eq4.3}
SCFE(\mathbb{X})=\int_{\mathbb I}\overline C_\mathbb{X}(\textbf{v})\big\{-\mathrm{Ln}_\eta\overline C_\mathbb{X}(\textbf{v})\big\}^{\frac{1}{\eta}}d\textbf{v},~~0<\eta<1.
\end{align} 
Note that the MSCRE always takes non-negative values for $0<\eta<1$. Next, we consider an example, dealing with FGM and AHM copulas.
\begin{example}\label{ex4.1}
Consider the Frank copula and Joe copula with their respective survival copula functions 
\begin{align*}
	\overline{C}(u,v)=u+v-1-\frac{1}{\theta}\log\left(1+\frac{(\mathrm{e}^{-\theta(1-u)}-1)(\mathrm{e}^{-\theta(1-v)}-1)}{\mathrm{e}^{-\theta}-1}\right), \theta\in \mathbb{R}-\{0\}
\end{align*}
and 
\begin{align*}
	\overline C(u,v)=u+v-\left[u^\theta+v^\theta-u^\theta v^\theta\right]^{\frac{1}{\theta}}, \theta\geq 1,
\end{align*}
respectively.
The MSCFI measure in (\ref{eq4.1}) is  obtained as
\begin{align}\label{eq4.4*}
SCFI(\textbf{X},\textbf{Y})&=\int_{0}^{1}\int_{0}^{1}\left(u+v-1-\frac{1}{\theta}\log\left(1+\frac{(\mathrm{e}^{-\theta(1-u)}-1)(\mathrm{e}^{-\theta(1-v)}-1)}{\mathrm{e}^{-\theta}-1}\right)\right)\nonumber\\
&\times\Big(-\mathrm{Ln}_\eta \left[ u^{\mu_1}+v^{\mu_2}-\left[u^{\mu_1\theta}+v^{\mu_2\theta}-u^{\mu_1\theta} v^{\mu_2\theta}\right]^{\frac{1}{\theta}}\right]\Big)^{\frac{1}{\eta}}dvdu\nonumber\\
&\approx (\eta!)^{\frac{1}{\eta}}\int_{0}^{1}\int_{0}^{1}\left(u+v-1-\frac{1}{\theta}\log\left(1+\frac{(\mathrm{e}^{-\theta(1-u)}-1)(\mathrm{e}^{-\theta(1-v)}-1)}{\mathrm{e}^{-\theta}-1}\right)\right)\nonumber\\
&\times\Big(-\log\left[ u^{\mu_1}+v^{\mu_2}-\left[u^{\mu_1\theta}+v^{\mu_2\theta}-u^{\mu_1\theta} v^{\mu_2\theta}\right]^{\frac{1}{\theta}}\right]\Big)^{\frac{1}{\eta}}dvdu
\end{align}
					\begin{figure}[] 
	\centering
\includegraphics{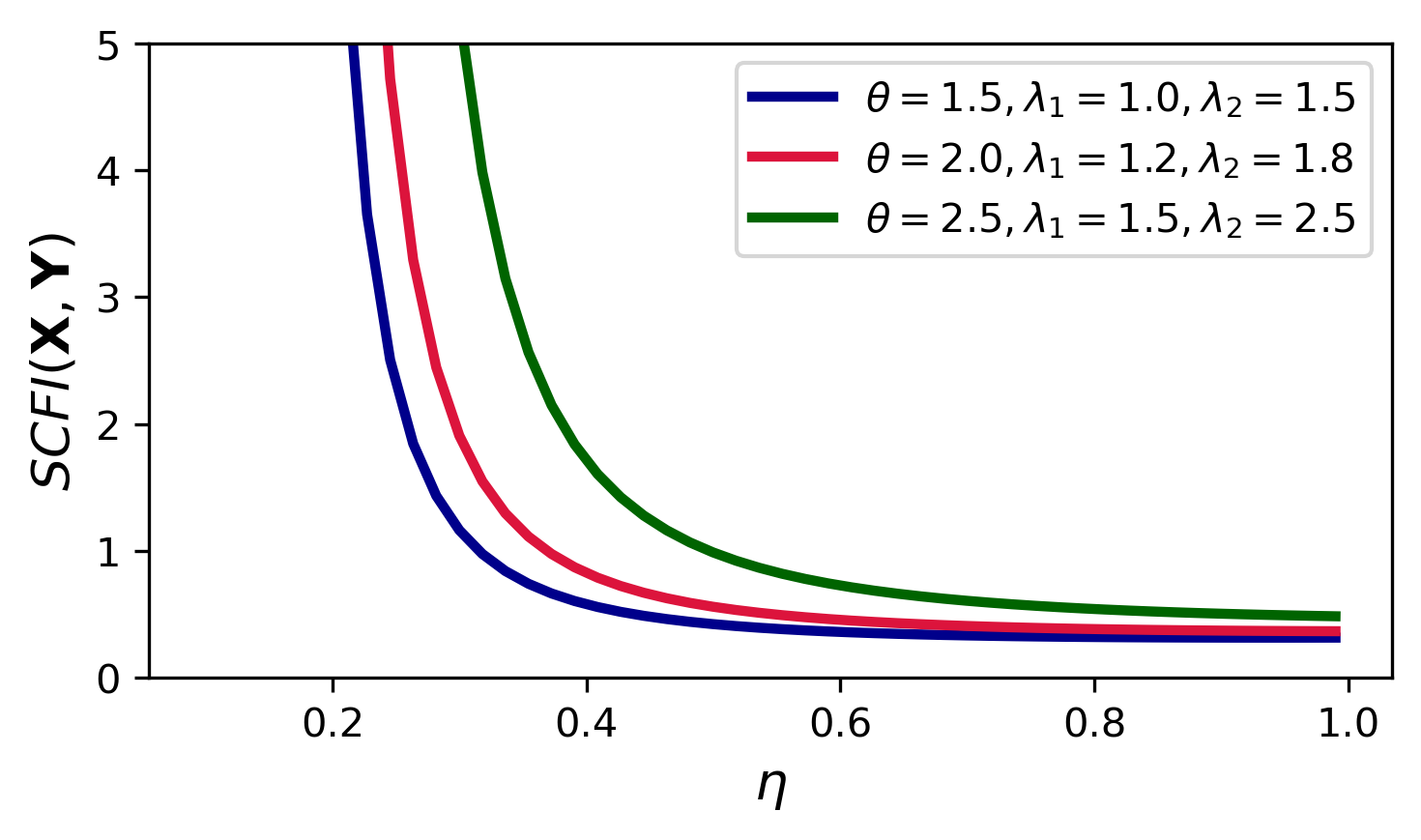}
	\caption{Plot of $CCFI(\textbf{X},\textbf{Y})$ (left) and $\widetilde{CCFI}(\textbf{X},\textbf{Y})$ (right) for Example 3.1.}\label{fig4.1}
	\label{fig1.3.4}
\end{figure}
Note that it is difficult to obtain the explicit forms of the MSCFI measure in (\ref{eq4.4*}). Thus we plot the numerical integration of (\ref{fig4.1}) in Figure 2.
\end{example}
\begin{proof}
The proof is omitted since it is similar to that of Proposition \ref{prop3.1}.
\end{proof}

It is of interest if there is any relation between MSCFI and MCCFI measures. Here, we notice that under some certain conditions there is a relation between MSCFI measure in (\ref{eq2.1}) and MCCFI measure in (\ref{eq4.1}).
\begin{proposition}\label{prop4.2}
Suppose $\textbf{X}$ and $\textbf{Y}$ have a common support $S=(a_1-c_1,a_1+c_1)\times\cdots\times(a_n-c_n,a_n+c_n)$, where $c_i>0$ and $a_i\in\mathbb{R}$ for $i=1,\cdots,n\in\mathbb{N}$. Further, let $\textbf{X}$ and $\textbf{Y}$ be radially symmetric. Then,
\begin{align*}
SCFI(\textbf{X},\textbf{Y})=CCFI(\textbf{X},\textbf{Y}).
\end{align*}
\end{proposition}
\begin{proof}
Using the transformation $u_i=F_i(x_i)$ in (\ref{eq2.1}), we have
\begin{align}\label{eq4.4}
CCFI(\textbf{X},\textbf{Y\textbf{}})=&\ \int_{a_1-c_1}^{a_1+c_1}\cdots\int_{a_n-c_n}^{a_n+c_n}f_1(z_1)\cdots f_n(z_n)C_X(F_1(z_1),\cdots,F_n(z_n))\nonumber\\
&\times \{-\log C_Y(G_1(z_1),\cdots,G_n(z_n))\}^{\eta}dz_1\cdots dz_n.
\end{align}
Similarly, using the transformation $u_i=\overline F_i(z_i)$ in (\ref{eq4.1}), we have
\begin{align}\label{eq4.5}
SCFI(\textbf{X},\textbf{Y\textbf{}})&=\ \int_{a_1-c_1}^{a_1+c_1}\cdots\int_{a_n-c_n}^{a_n+c_n}f_1(z_1)\cdots f_n(z_n)\overline C_X(\overline F_1(z_1),\cdots,\overline F_n(z_n))\nonumber\\
&\times \{-\mathrm{Ln}_\eta \overline C_Y(\overline G_1(z_1),\cdots,\overline G_n(z_n))\}^{\frac{1}{\eta}}dz_1\cdots dz_n.
\end{align}
 Since, $\textbf{X}$ and $\textbf{Y}$ are radially symmetric, we have 
 \begin{align}\label{eq4.6}
 F(z_1,\cdots,x_n)=\overline F(2a_1-z_1,\cdots,2a_n-z_n)~\text{and}~G(z_1,\cdots,z_n)=\overline G(2a_1-z_1,\cdots,2a_n-z_n).
 \end{align}
Using (\ref{eq4.6}) in (\ref{eq4.4}), we obtain
\begin{align}\label{eq4.7}
CCFI(\textbf{X},\textbf{Y})&=\ \int_{a_1-c_1}^{a_1+c_1}\cdots\int_{a_n-c_n}^{a_n+c_n}f_1(z_1)\cdots f_nz_n\overline F(2a_1-z_1,\cdots,2a_n-z_n)\nonumber\\
&\times \{-\mathrm{Ln}_\eta\overline G(2a_1-z_1,\cdots,2a_n-z_n)\}^{\frac{1}{\eta}}dz_1\cdots dz_n.
\end{align}
Taking $y_i=2a_i-x_i$, then from (\ref{eq4.7}), we get
\begin{align*}
CCFI(\textbf{X},\textbf{Y})&=\ \int_{a_1-c_1}^{a_1+c_1}\cdots\int_{a_n-c_n}^{a_n+c_n}f_1(y_1)\cdots f_n(y_n)\overline F(y_1,\cdots,y_n)\nonumber\\
&\times \{-\mathrm{Ln}_\eta\overline G(y_1,\cdots,y_n)\}^{\frac{1}{\eta}}dy_1\cdots dy_n\nonumber\\
&=SCFI(\textbf{X},\textbf{Y}).
\end{align*}
Therefore, the proof is finished.
\end{proof}
The comparison of two multivariate statistical inaccuracy measures are very important to select a better model. In the following we discuss the comparison study for two MSCFI measures. 

\begin{proposition}\label{prop4.3}
Suppose $\textbf{X}$, $\textbf{Y}$ and $\textbf{Z}$ have survival copula functions $\overline C_\textbf{X}$, $\overline C_\textbf{Y}$ and $\overline C_\textbf{Z}$, respectively. Assume that $\overline H_i(x_i)=\overline F_i^{\gamma_i}(x_i)$ and $\overline C_{\textbf{X}}=\overline C_{\textbf{Z}}$,~$i=1,\cdots,n$. Then, 
Then, for \(0 < \eta < 1\), we have  
\begin{align*}
	\text{SCFI}(\mathbf{Z}, \mathbf{Y}) &\ge \text{SCFI}(\mathbf{X}, \mathbf{Y})~~for~~\gamma_i>1\\
	\text{SCFI}(\mathbf{Z}, \mathbf{Y}) &\le \text{SCFI}(\mathbf{X}, \mathbf{Y})~~for~~0<\gamma_i<1; i=1,2,\cdots,n.
\end{align*}  
\end{proposition}
\begin{proof}
The proof is similar to Proposition \ref{prop3.2}. Hence, it is omitted.
\end{proof}

\begin{proposition}
Consider  $\textbf{X}$, $\textbf{Y}$ and $\textbf{Z}$ with survival copula functions $\overline C_\textbf{X},~\overline C_\textbf{Y}$ and $\overline C_\textbf{Z}$, respectively. Assume $\overline C_\textbf{Z}=\overline C_\textbf{Y}$, $\overline G_i=\overline F_i^{\gamma_i}$ and $\overline H_i=\overline F_i^{\delta_i}$ with $\gamma_i<\delta_i,~i=1,\cdots,n\in\mathbb{N}.$   If $\gamma>1$, then $SCFI(\mathbb{X},\mathbb{Y})\le SCFI(\mathbb{X},\mathbb{Z})$.
\end{proposition}
\begin{proof}
The proof is similar to that of Proposition \ref{prop3.3}. 
\end{proof}
The upper orthant order can reveal how much of this shared information is concentrated in extreme upper-tail events. In climate modeling, analyzing upper-tail information measure helps understand dependencies during simultaneous extreme weather conditions. 
\begin{proposition}\label{prop4.5}
If $\textbf{X}\le_{UO}\textbf{Y}$ for $0<\eta<1$, then 
 $$SCFI(\textbf{Z},\textbf{X})\ge SCFI(\textbf{Z},\textbf{Y}).$$
\end{proposition}
\begin{proof}
We have $\textbf{X}\le_{UO}\textbf{Y}$, implying that $\overline F(z_1,\cdots,x_n)\le \overline G(z_1,\cdots,x_n).$ Then, from Sklar's Theorem, we get
\begin{align}\label{eq4.12}
\overline C_\textbf{X}\big(\overline F_1(z_1),\cdots, \overline F_nz_n\big)\le \overline C_\textbf{Y}\big(\overline G_1(z_1),\cdots, \overline G_nz_n\big).
\end{align}
Employing the transformation $x_i=\overline H^{-1}_i(v_i)$ for $i=1,\cdots,n\in\mathbb{N}$ in (\ref{eq4.12}), we obtain
\begin{align}\label{eq4.13}
\overline C_\textbf{X}\big(\overline F_1(\overline H^{-1}_1(v_1)),\cdots, \overline F_n(\overline H^{-1}_n(v_n))\big)\le \overline C_\textbf{Y}\big(\overline G_1(\overline H^{-1}_1(v_1)),\cdots, \overline G_n(\overline H^{-1}_n(v_n))\big).
\end{align} From (\ref{eq4.13})
\begin{align}\label{eq4.16*}
&\Big\{-\mathrm{Ln}_\eta\overline C_\textbf{X}\big(\overline F_1(\overline H^{-1}_1(v_1)),\cdots, \overline F_n(\overline H^{-1}_n(v_n))\big)\Big\}^{\frac{1}{\eta}}\ge \Big\{-\mathrm{Ln}_\eta\overline C_\textbf{Y}\big(\overline G_1(\overline H^{-1}_1(v_1)),\cdots, \overline G_n(\overline H^{-1}_n(v_n))\big)\Big\}^{\frac{1}{\eta}}\nonumber\\
\end{align}
Now, multiplying by $\overline{C}_\textbf{Z}(v_1,\cdots,v_n)$ both sides and integrating over the region $\mathbb I$, yields
\begin{align*}
	SCFI(\textbf{Z},\textbf{X})\ge SCFI(\textbf{Z},\textbf{Y}).
\end{align*}
\end{proof}
\begin{proposition}\label{prop4.6}
Let $\textbf{X}\le_{UO}\textbf{Y}$. Further, let $X_i\overset{\mathrm{st}}{=}Y_i,~i=1,\cdots,n$. Then, $SCFI(\textbf{X},\textbf{Z})\le SCFI(\textbf{Y},\textbf{Z})$.
\end{proposition}
\begin{proof}
We have 
\begin{align}\label{eq4.13}
\textbf{X}\le_{UO}\textbf{Y}\Rightarrow \overline C_\textbf{X}\big(\overline F_1(z_1),\cdots, \overline F_nz_n\big)\le \overline C_\textbf{Y}\big(\overline G_1(z_1),\cdots, \overline G_nz_n\big).
\end{align}
Utilizing $X_i\overset{\mathrm{st}}{=}Y_i$ and applying the transformation $u_i=\overline F_i(x_i),~i=1,\cdots,n$, in (\ref{eq4.13}), we get 
\begin{align*}
\overline C_\textbf{X}\big(v_1,\cdots, v_n\big)\le \overline C_\textbf{Y}\big(v_1,\cdots, v_n\big).
\end{align*} Now, multiplying by $\Big\{-\mathrm{Ln}_\eta\overline C_Z\big(\overline H_1(\overline F_1^{-1}(v_1)),\cdots,\overline H_n(\overline F_n^{-1}(v_n))\big)\Big\}^{\frac{1}{\eta}}$ both sides. We get,
\begin{align*}
\overline C_\textbf{X}\big(v_1,\cdots, v_n\big)	&\Big\{-\mathrm{Ln}_\eta\overline C_Z\big(\overline H_1(\overline F_1^{-1}(v_1)),\cdots,\overline H_n(\overline\nonumber F_n^{-1}(v_n))\big)\Big\}^{\frac{1}{\eta}}\leq\\& C_\textbf{Y}\big(v_1,\cdots, v_n\big)\Big\{-\mathrm{Ln}_\eta\overline C_Z\big(\overline H_1(\overline F_1^{-1}(v_1)),\cdots,\overline H_n(\overline F_n^{-1}(v_n))\big)\Big\}^{\frac{1}{\eta}}
\end{align*}
Equivalently,
\begin{align}\label{4.16}
	\overline C_\textbf{X}\big(v_1,\cdots, v_n\big)	&\Big\{-\mathrm{Ln}_\eta\overline C_Z\big(\overline H_1(\overline F_1^{-1}(v_1)),\cdots,\overline H_n(\overline\nonumber F_n^{-1}(v_n))\big)\Big\}^{\frac{1}{\eta}}\leq\\& C_\textbf{Y}\big(v_1,\cdots, v_n\big)\Big\{-\mathrm{Ln}_\eta\overline C_Z\big(\overline H_1(\overline G_1^{-1}(v_1)),\cdots,\overline H_n(\overline G_n^{-1}(v_n))\big)\Big\}^{\frac{1}{\eta}}.
\end{align}
\end{proof}

\begin{proposition}\label{prop4.7}
Suppose that $\mathbb{X},~\mathbb{Y}$ and $\mathbb{Z}$ have copulas $C_\mathbb{X},~C_\mathbb{Y}$ and $C_\mathbb{Z}$, respectively. 
\begin{enumerate}
	\item If $\mathbb{Z}\le_{LO}\mathbb{Y},~\mathbb{Z}\le_{LO}\mathbb{X}$ and $Z_i\overset{\mathrm{st}}{=}X_i$,~$i=1,\cdots,n$. 
	Then, $SCFI(\mathbb X,\mathbb Z)\leq SCFI(\mathbb X,\mathbb Y)\le CCFI(\mathbb Z,\mathbb Y).$;
	\item  If $\mathbb{X}\le_{LO}\mathbb{Z}\le_{LO}\mathbb{Y}$ and $Z_i\overset{\mathrm{st}}{=}X_i$,~$i=1,\cdots,n$. \\Then,
	$SCFI(\mathbb{X},\mathbb{Y})\ge \max\{SCFI(\mathbb{Z},\mathbb{Y})\le CCFI(\mathbb{X},\mathbb{Z})\}$;
	\item If $\mathbb{Y}\le_{LO}\mathbb{Z},~\mathbb{X}\le_{LO}\mathbb{Z}$ and$Z_i\overset{\mathrm{st}}{=}X_i$,~$i=1,\cdots,n$. \\Then,
	$SCFI(\mathbb Z,\mathbb Y)\leq SCFI(\mathbb{X},\mathbb{Y})\ge SCFI(\mathbb X,\mathbb Z)$.
\end{enumerate}
\end{proposition}

\begin{proof}
The argument mirrors that of Proposition \ref{prop3.5}, so we omit the proof here.
\end{proof}
\section{Multivariate co-copula and dual copula fractional inaccuracy measures}\label{sec5}
In many fields such as reliability engineering, survival analysis, and insurance, probabilities like
$
P(X_1 > x_1 ~\text{or}~ \cdots ~\text{or}~ X_n > x_n) \quad \text{and} \quad P(X_1 < x_1 ~\text{or}~ \cdots ~\text{or}~ X_n < x_n)
$
are crucial, especially in systems with dependent, load-sharing components. In such setups, component failures affect the remaining components' reliability, often leading to survival or cascading failure scenarios.
These probabilities are effectively modeled using copula theory: particularly co-copula and dual copula functions, which account for dependence without assuming identical distributions. Motivated by their relevance, we introduce two new multivariate fractional inaccuracy measures  based on these copula structures and explore their key theoretical properties.

Assume that $\mathbb{X}$ and $\mathbb{Y}$ are two RVs with multivariate co-copula functions $\widehat{C}_{\mathbb{X}}$ and $\widehat{C}_{\mathbb{Y}}$, and dual copulas $\widetilde C_{\mathbb{X}}$ and $\widetilde C_{\mathbb{Y}}$, respectively. Then, the multivariate co-copula fractional inaccuracy (MCoCFI) and multivariate dual copula fractional inaccuracy (MDCFI) measures are defined as
\begin{align}\label{eq5.1}
CoCFI(\mathbb{X},\mathbb{Y})=\int_{\mathbb I}\widehat{C}_\mathbb{X}(\textbf{v})\{-\mathrm{Ln}_\eta\widehat{C}_\mathbb{Y}\big(\mathbb{G}(\mathbb{F}^{-1}(\textbf{v}))\big\}^{\frac{1}{\eta}}d\textbf{v},~~0<\eta<1
\end{align}
and 
\begin{align}\label{eq5.2}
DCCFI(\mathbb{X},\mathbb{Y})=\int_{\mathbb I}\widetilde C_\mathbb{X}(\textbf{v})\{-\mathrm{Ln}_\eta\widetilde C_\mathbb{Y}\big(\mathbb{G}(\mathbb{F}^{-1}(\textbf{v}))\big\}^{\frac{1}{\eta}}d\textbf{v},~~0<\eta<1,
\end{align}
 respectively, where $\textbf{v}=(v_1,v_2,\cdots,v_n)$ and $\mathbb{G}(\mathbb{F}^{-1}(\textbf{v}))=\big(G_1(F_1^{-1}(v_1)),\cdots,G_n(F_n^{-1}(v_n))\big)$.
Next, we discuss an example to study the behaviour of the  proposed MCoCFI and MDCFI measures considering Joe and AHM copulas.
\begin{example}\label{ex5.1}
Suppose $\textbf{X}=(X_1,X_2)$ and $\textbf{Y}=(Y_1,Y_2)$ are two bivariate RVs with their copula functions 
\begin{align*}
 C_{\textbf{X}}(u,v)=-\frac{1}{\theta}\log\left(1+\frac{\left(\mathrm{e}^{-\theta u}-1\right)\left(\mathrm{e}^{-\theta v}-1\right)}{\mathrm{e}^{-\theta}-1}\right)
\end{align*}
and 
 \begin{align*}
 C_{\textbf{Y}}(u,v)= C_{\textbf{X}}(u,v)=1-\left((1-u)^\theta+(1-v)^\theta-(1-u)^\theta(1-v)^\theta\right)^{\frac{1}{\theta}},
\end{align*}
respectively. Further, assume that $X_1$ and $X_2$ are two standard exponential rvs and $Y_1$ and $Y_2$ are two rvs of exponential distributions with parameters $\mu_1$ and $\mu_2$, respectively. Therefore, the MCoCFI measure in (\ref{eq5.1}) and MDCFI measure in (\ref{eq5.2}) are, respectively obtained as
\begin{align}\label{eq5.5*}
CoCFI(\textbf{X},\textbf{Y})&=\int_{0}^{1}\int_{0}^{1}\left(1 +\frac{1}{\theta}\log\left(1+\frac{\left(\mathrm{e}^{-\theta (1-u)}-1\right)\left(\mathrm{e}^{-\theta (1-v)}-1\right)}{\mathrm{e}^{-\theta}-1}\right)\right)\nonumber\\
&\times\Big(-\mathrm{Ln}_\eta\big((1-(1-u)^{\mu_1})^\theta+(1-(1-v)^{\mu_2})^\theta\nonumber\\&~~~-(1-(1-u)^{\mu_1})^\theta (1-(1-v)^{\mu_2})^\theta\big)^{\frac{1}{\theta}}\Big)^{\frac{1}{\eta}}\nonumber\\
&\approx (\eta !)^{\frac{1}{\eta}}\int_{0}^{1}\int_{0}^{1} \left( 1+\frac{1}{\theta}\log\left(1+\frac{\left(\mathrm{e}^{-\theta (1-u)}-1\right)\left(\mathrm{e}^{-\theta (1-v)}-1\right)}{\mathrm{e}^{-\theta}-1}\right)\right)\nonumber\\
&~~~\times\Big(-\log\big((1-(1-u)^{\mu_1})^\theta+(1-(1-v)^{\mu_2})^\theta\nonumber\\&~~~-(1-(1-u)^{\mu_1})^\theta (1-(1-v)^{\mu_2})^\theta\big)^{\frac{1}{\theta}}\Big)^{\frac{1}{\eta}}dvdu
\end{align}
and 
\begin{align}\label{eq5.6*}
DCFI(\textbf{X},\textbf{Y})=&\int_{0}^{1}\int_{0}^{1}\left[u+v+\frac{1}{\theta}\log\left(1+\frac{\left(\mathrm{e}^{-\theta u}-1\right)\left(\mathrm{e}^{-\theta v}-1\right)}{\mathrm{e}^{-\theta}-1}\right)\right]\nonumber\\
&\times\Big[-\mathrm{Ln}_\eta\big[1-(1-u)^{\mu_1}-(1-v)^{\mu_2}\\&+\left((1-u)^{\mu_1\theta}+(1-v)^{\mu_2\theta}-(1-u)^{\mu_1\theta}(1-v)^{\mu_2\theta}\right)^{\frac{1}{\theta}}\big]\Big]^{\frac{1}{\eta}}dvdu.
\end{align}
It is challenging to derive closed-form expressions for the MCoCFI and MDCFI measures. Therefore, we illustrate these measures graphically in Figure 3 to examine their behavior with respect to the parameters $\theta$, $\gamma$, $\mu_1$, and $\mu_2$. As the values of the parameters increase, both CoCFI and DCFI also increase.
					\begin{figure}[] \label{fig5.3*}
	\centering
	\begin{minipage}[b]{0.44\linewidth}
		\includegraphics[height=7cm,width=7.7cm]{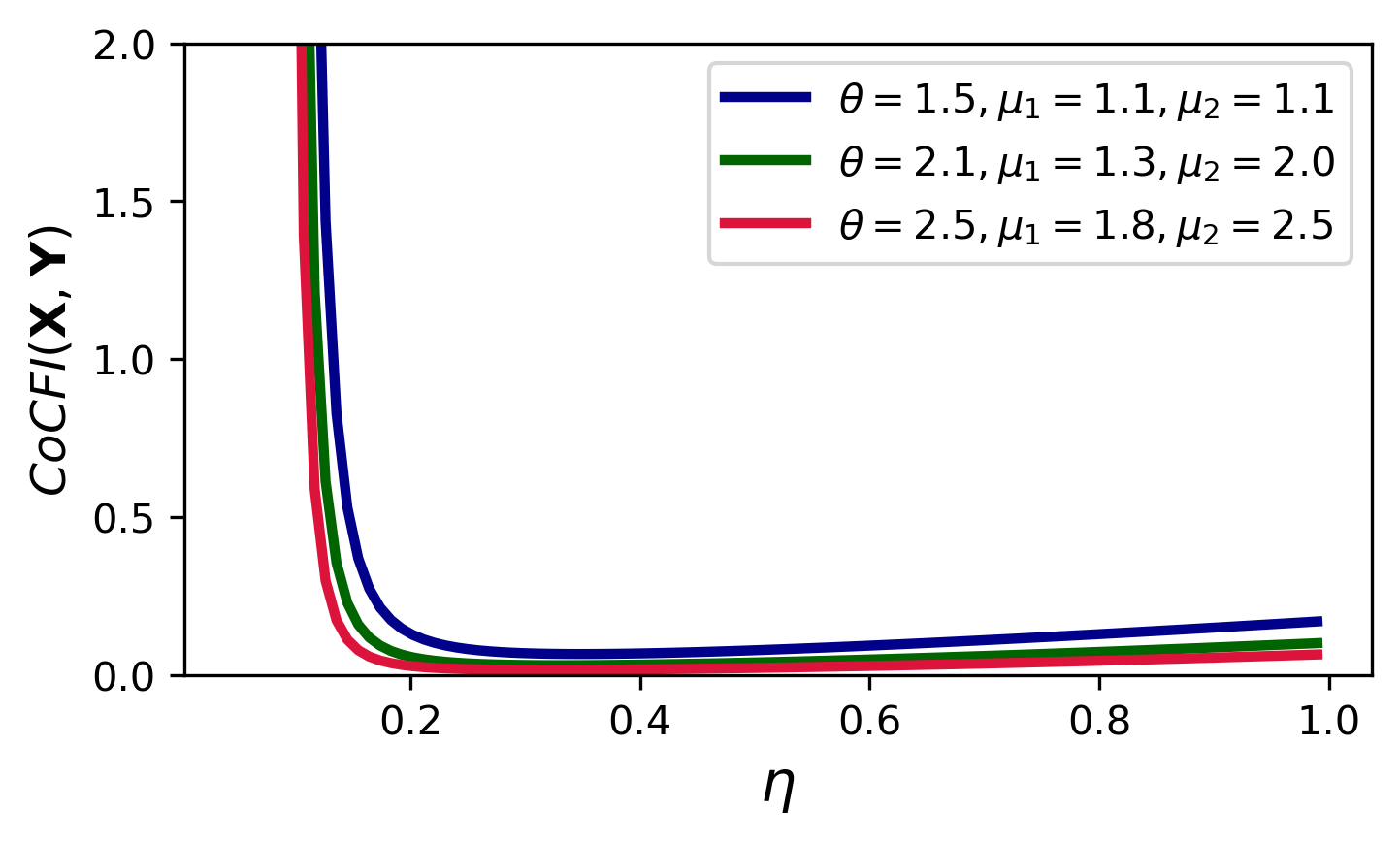}
		\centering{(a)  }
		\label{fig1.3.3}
	\end{minipage}
	\quad
	\begin{minipage}[b]{0.44\linewidth}
		\includegraphics[height=6.88cm,width=8.5cm]{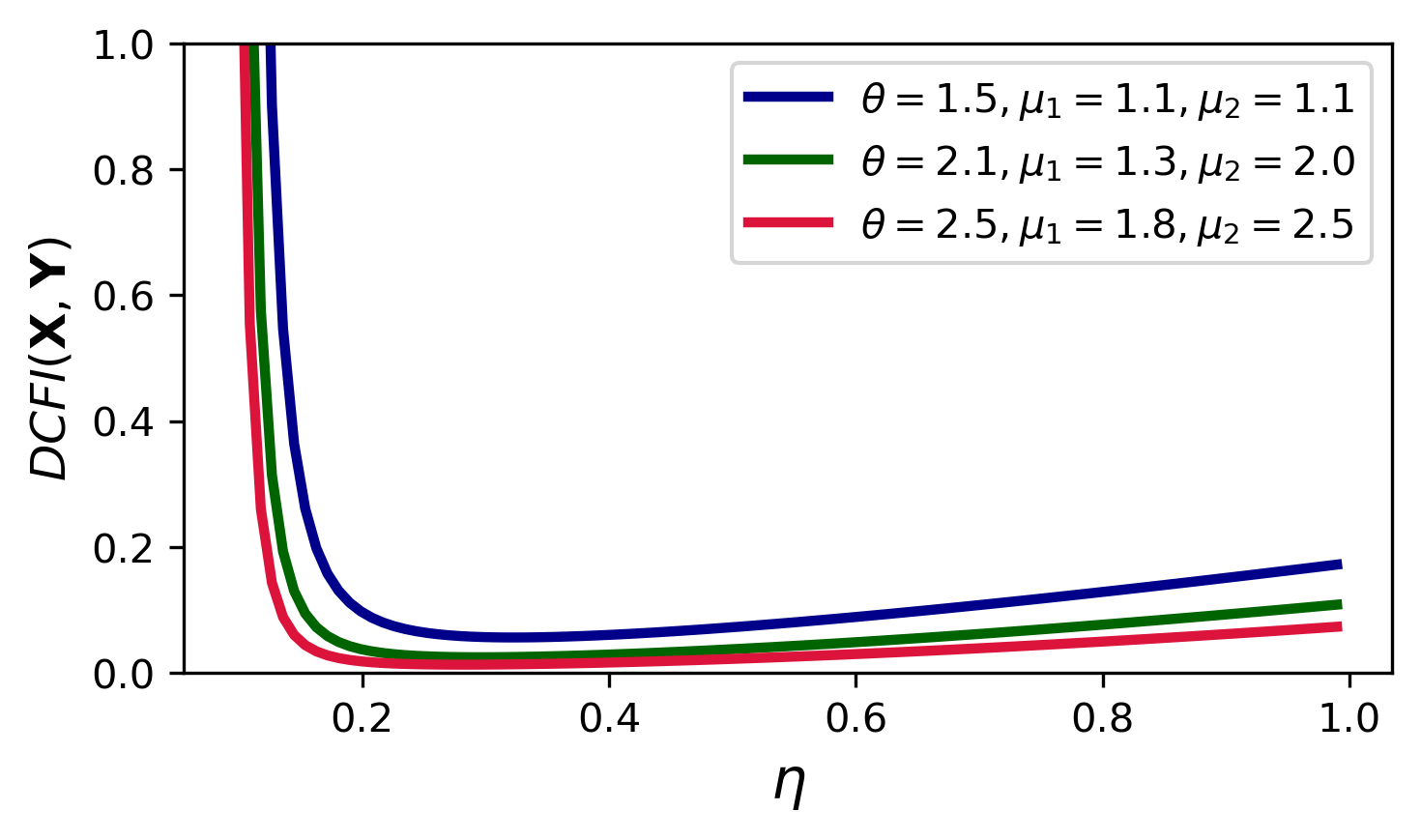}
		\centering{(b)}
	\end{minipage}
	\caption{Plot of $CoCFI(\textbf{X},\textbf{Y})$ (left) and ${DCFI}(\textbf{X},\textbf{Y})$ (right) for Example 5.1.}\label{fig5.3}
	\label{fig1.3.4}
\end{figure}
\end{example}

Next, we propose some results for bivariate RVs similar to the measures $MCCFI$ and $MSCFI$.

\begin{proposition}
	Suppose  $\mathbb{X}$, $\mathbb{Y}$ and $\mathbb{Z}$ have co-copulas $\widehat{C}_\mathbb{X},~\widehat{C}_\mathbb{Y}$ and $\widehat{C}_\mathbb{Z}$, respectively. Assume that $\overline{F}_i,~\overline{G}_i$ and $\overline{H}_i,~i=1,\cdots,n\in\mathbb{N}$ are the SFs of $X_i,~Y_i$ and $Z_i$, respectively. If $\mathbb{X}\le_{LO}\mathbb{Y}$, then
 $CoCFI(\mathbb{Z},\mathbb{X})\le CoCFI(\mathbb{Z},\mathbb{Y})$.
\end{proposition}
\begin{proof}
	The proof is similar to that of Proposition \ref{prop3.4}. 
\end{proof}

\begin{proposition}
	Let $\mathbb{X}\le_{LO}\mathbb{Y}$. Further, let $X_i\overset{\mathrm{st}}{=}Y_i$,~$i=1,\cdots,n$. Then,
$CoCFI(\mathbb{X},\mathbb{Z})\ge CoCFI(\mathbb{Y},\mathbb{Z})$.
\end{proposition}
\begin{proof}
	The proof is similar to that of Proposition \ref{prop3.5}. 
\end{proof}
\begin{proposition}
	Suppose  $\mathbb{X}$, $\mathbb{Y}$ and $\mathbb{Z}$ have dual copula functions $\widetilde C_\mathbb{X},~\widetilde C_\mathbb{Y}$ and $\widetilde C_\mathbb{Z}$, respectively. Assume that $ F_i,~ G_i$ and $ H_i$ are the CDFs of $X_i,~Y_i$ and $Z_i$, for $i=1,\cdots,n\in\mathbb{N}$, respectively. If $\textbf{X}\le_{UO}\textbf{Y}$ for $0<\gamma\ne1$, then 
	$$DCFI(\mathbb{Z},\mathbb{X})\le DCFI(\mathbb{Z},\mathbb{Y}).$$
\end{proposition}
\begin{proposition}
	Let $\mathbb{X}\le_{UO}\mathbb{Y}$. Further, let $X_i\overset{\mathrm{st}}{=}Y_i,~i=1,\cdots,n$. Then,
$DCFI(\mathbb{X},\mathbb{Z})\le DCFI(\mathbb{Y},\mathbb{Z})$.
\end{proposition}
\section*{Conflicts of interest} The authors declare there is no conflict of interest.


\begin{thebibliography}{99}
	
	\bibitem{1} Ahmadi, J., Di Crescenzo, A., Longobardi, M.: On dynamic mutual information for bivariate life times. Adv. Appl. Probab. \textbf{47}, 1157–1174 (2015)
	
	\bibitem{2} Di Crescenzo, A., Longobardi, M.: On cumulative entropies. J. Stat. Plann. Inference \textbf{139}(12), 4072–4087 (2009)
	
	\bibitem{3} Fallah Mortezanejad, S.A., Mohtashami Borzadaran, G., Sadeghpour Gildeh, B.: Joint dependence distribution of data set using optimizing Tsallis copula entropy. Physica A \textbf{533}, 121897 (2019)
	
	\bibitem{4} Gronneberg, S., Hjort, N.L.: The copula information criteria. Scand. J. Stat. \textbf{41}, 436–459 (2014)
	
	\bibitem{5} Gupta, R.C., Gupta, R.D.: Proportional reversed hazard rate model and its applications. J. Stat. Plann. Inference \textbf{137}(11), 3525–3536 (2007)
	
	\bibitem{6} Haubold, H.J., Mathai, A.M., Saxena, R.K.: Mittag-Leffler functions and their applications. J. Appl. Math. \textbf{2011}(1), 298628 (2011)
	
	\bibitem{7} Hosseini, T., Ahmadi, J.: Results on inaccuracy measure in information theory based on copula function. In: Proc. 5th Seminar of Copula Theory and its Applications, pp. 22–32 (2019)
	
	\bibitem{8} Hosseini, T., Nooghabi, M.J.: Discussion about inaccuracy measure in information theory using co-copula and copula dual functions. J. Multivar. Anal. \textbf{183}, 104725 (2021)
	
	\bibitem{9} Jumarie, G.: Derivation of an amplitude of information in the setting of a new family of fractional entropies. Inf. Sci. \textbf{216}, 113–137 (2012)
	
	\bibitem{10} Kerridge, D.F.: Inaccuracy and inference. J. R. Stat. Soc. Ser. B Stat. Methodol. \textbf{23}(2), 184–194 (1961)
	
	\bibitem{11} Kharazmi, O., Contreras-Reyes, J.E.: Fractional cumulative residual inaccuracy information measure and its extensions with application to chaotic maps. Int. J. Bifurc. Chaos \textbf{34}(01), 2450006 (2024)
	
	\bibitem{12} Kullback, S., Leibler, R.A.: On information and sufficiency. Ann. Math. Stat. \textbf{22}(1), 79–86 (1951)
	
	\bibitem{13} Ma, J., Sun, Z.: Mutual information is copula entropy. Tsinghua Sci. Technol. \textbf{16}(1), 51–54 (2011)
	
	\bibitem{14} Mittag-Leffler, G.M.: Sur la nouvelle fonction $E_\gamma (x)$. C. R. Acad. Sci. Paris \textbf{137}(2), 554–558 (1903)
	
	\bibitem{15} Nelsen, R.B.: An Introduction to Copulas. Springer, Berlin (2006)
	
	\bibitem{16} Pougaza, D.B., Mohammad-Djafari, A.: Maximum entropies copulas. In: AIP Conf. Proc. \textbf{1305}(1), 329–336 (2011)
	
	\bibitem{17} Preda, V., Sfetcu, R.C., Sfetcu, S.C.: Some generalizations concerning inaccuracy measures. Results Math. \textbf{78}, 195 (2023)
	
	\bibitem{18} Saha, S., Kayal, S.: Copula-based extropy measures, properties, and dependence in bivariate distributions. Commun. Stat. Theory Methods \textbf{1}–28 (2025)
	
	\bibitem{19} Saha, S., Kayal, S.: Fractional cumulative past inaccuracy measure, its dynamic version and applications in survival analysis. Physica D \textbf{473}, 134-545 (2025)
	
	\bibitem{20} Shannon, C.E.: A mathematical theory of communication. Bell Syst. Tech. J. \textbf{27}, 379–423 (1948)
	
	\bibitem{21} Singh, V.P., Zhang, L.: Copula–entropy theory for multivariate stochastic modeling in water engineering. Geosci. Lett. \textbf{5}, 6 (2018)
	
	\bibitem{22} Zhang, B., Shang, P.: Cumulative permuted fractional entropy and its applications. IEEE Trans. Neural Netw. Learn. Syst. \textbf{32}(11), 4946–4955 (2021)
	
	\bibitem{23} Zhao, N., Lin, W.T.: A copula entropy approach to correlation measurement at the country level. Appl. Math. Comput. \textbf{218}(2), 628–642 (2011)
	
\end{thebibliography}
\end{document}